\documentclass[12pt,english,reqno]{amsart}
\usepackage[T1]{fontenc}
\usepackage[utf8]{inputenc}
\usepackage[a4paper]{geometry}
\usepackage{dsfont}
\usepackage{amstext}
\usepackage{amsthm}
\usepackage{amssymb}
\usepackage{stackrel}
\usepackage{graphicx}
\usepackage{setspace}
\makeatletter

\providecommand{\tabularnewline}{\\}

\numberwithin{equation}{section}
\numberwithin{figure}{section}
\theoremstyle{plain}
\newtheorem{thm}{\protect\theoremname}[section]
\theoremstyle{definition}
\newtheorem{defn}[thm]{\protect\definitionname}
\theoremstyle{plain}
\newtheorem{conjecture}[thm]{\protect\conjecturename}
\theoremstyle{plain}
\newtheorem{prop}[thm]{\protect\propositionname}
\theoremstyle{plain}
\newtheorem{lem}[thm]{\protect\lemmaname}
\theoremstyle{plain}
\newtheorem{cor}[thm]{\protect\corollaryname}
\theoremstyle{remark}
\newtheorem{observation}[thm]{\protect\observationname}
\theoremstyle{remark}
\newtheorem{rem}[thm]{\protect\remarkname}

\usepackage{amsmath}
\usepackage{bbm}
\usepackage{dsfont}

\usepackage[unicode=true,pdfusetitle,
 bookmarks=true,bookmarksnumbered=false,bookmarksopen=false,
 breaklinks=false,pdfborder={0 0 0},pdfborderstyle={},backref=false,colorlinks=false]
 {hyperref}

\allowdisplaybreaks
\usepackage{tikz}

\usepackage{etoolbox}
\patchcmd{\@settitle}{\uppercasenonmath\@title}{}{}{}
\patchcmd{\@setauthors}{\MakeUppercase}{}{}{}
\patchcmd{\section}{\scshape}{}{}{}

\makeatother

\usepackage{babel}
\providecommand{\conjecturename}{Conjecture}
\providecommand{\corollaryname}{Corollary}
\providecommand{\definitionname}{Definition}
\providecommand{\lemmaname}{Lemma}
\providecommand{\observationname}{Observation}
\providecommand{\propositionname}{Proposition}
\providecommand{\remarkname}{Remark}
\providecommand{\theoremname}{Theorem}

\begin{document}
\global\long\def\One{\mathds{1}}%

\global\long\def\bone{\mathds{\mathbf{1}}}%

\global\long\def\Laplacian{\Delta}%

\global\long\def\grad{\nabla}%

\global\long\def\norm#1{\left\Vert #1\right\Vert }%

\global\long\def\Z{\mathbb{Z}}%

\global\long\def\R{\mathbb{R}}%

\global\long\def\N{\mathbb{N}}%

\global\long\def\P{\mathbb{P}}%

\global\long\def\E{\mathbb{E}}%

\global\long\def\cL{\mathcal{L}}%

\global\long\def\cD{\mathcal{D}}%

\global\long\def\floor#1{\left\lfloor #1\right\rfloor }%

\global\long\def\ceil#1{\left\lceil #1\right\rceil }%

\global\long\def\var{\operatorname{Var}}%

\global\long\def\cov{\operatorname{Cov}}%

\global\long\def\dd{\operatorname{d}}%

\global\long\def\connected#1#2{\stackrel[#2]{#1}{\longleftrightarrow}}%

\title{Universality in dimension $1$ of kinetically constrained lattice
gases}
\author{Assaf Shapira}
\address{Université Paris Cité, CNRS, MAP5, F-75006 Paris, France.}
\email{assaf.shapira@math.cnrs.fr}
\urladdr{https://assafshap.github.io/}
\begin{abstract}
Kinetically constrained lattice gases are interacting particle systems
with conserved number of particles, having degenerate rates caused
by a kinetic constraint. Over the past decade, there has been major
progress in the study of their non-conservative counterpart, \emph{kinetically
constrained spin models}: in dimensions 1 and 2, we have a good understanding
of the universality classes of these models. However, the conservative
systems are much more challenging to analyze, and very few results
are available. The purpose of this paper is to describe universality
in the one dimensional case. We will characterize the three universality
classes, determining whether such a model is always ergodic, never
ergodic, or exhibits an ergodicity phase transition.
\end{abstract}

\maketitle

\section{Introduction}

Kinetically Constrained Lattices Gases (KCLGs) form a family of interacting
particle systems with conserved number of particles, originally introduced
in the physics literature to model glasses (see, e.g., \cite{KobAndersen,RitortSollich}).
They can be seen as conservative versions of \emph{kinetically constrained
spin models} thoroughly studied in both mathematics and physics \cite{HartarskyToninelli2024kcmbook}.
These models are defined by adding \emph{constraints} to the dynamics:
transitions are only allowed when a certain condition is satisfied.
The details of the constraint depend on the exact model; in general
it requires sufficiently many sites in the vicinity to be empty. This
reflects the behavior of glasses, where motion is suppressed at high
densities.

The question of universality is of cardinal importance when studying
models in statistical physics: can we identify a few properties of
a model that completely determine the important features of its behavior?

In the case of (non-conservative) kinetically constrained spin models
in dimensions $1$ and $2$, the answer is yes: all such models can
be divided into a handful of universality classes, which determine
their large-scale behavior \cite{HartarskyToninelli2024kcmbook}.
In particular, given a kinetically constrained spin model, a quick
calculation can determine its ergodicity depending on the particle
density.

Much less is known in the case of KCLGs: there is one class of models,
called \emph{noncooperative}, which is relatively well understood
\cite{Shapira2024Noncooperative}, but it does not show the rich phenomenology
observed in other, \emph{cooperative}, KCLGs. Some results are available
for one specific cooperative model, called the Kob-Andersen model
\cite{ToninelliBiroliFisher,CMRT2010KCLG,MartinelliShapiraToninelli2020KA_gap,Shapira23KA_HL},
but no general theory is available.

This article focuses on one-dimensional models, and on the question
of ergodicity phase transition: what is the critical density, above
which the system is non-ergodic (and in particular some particles
never move), and below which it is ergodic (so all particles eventually
move)? We will identify three universality classes, determining whether
this critical density is $0$, $1$, or strictly between $0$ and
$1$.

After understanding ergodicity, we discuss diffusivity: in conservative
interacting particle systems, the \emph{diffusion coefficient} describes
the evolution of the density profile at large scale, under diffusive
scaling (i.e., when scaling $\text{time}\propto\text{space}^{2}$).
In particular, positive diffusion coefficient $D$ indicates that
on a length scale $L$ the density evolves at time scale $L^{2}/D$;
while zero diffusion coefficient indicates that the slowing down of
the dynamics caused by the constraint is strong enough to make this
evolution \emph{sub-diffusive}, taking time much longer than $L^{2}$
(possibly infinite).

We begin in Section \ref{sec:Setting_and_results} with some definitions
and statement of main results. In Sections \ref{sec:east_or} and
\ref{sec:east_and} we discuss two specific models (that can be seen
as conservative versions of the East model, as well as variations
of the model studied in \cite{ScheutzZahra2026conservative_east}),
and identify their critical parameter. These are then used in Section
\ref{sec:proof_irreducibility}, where we prove the universality result.
Section \ref{sec:proof_D} is dedicated to diffusivity of KCLGs. We
conclude in Section \ref{sec:SOC} with a remark on the link between
dynamical phase transition and self-organized criticality, describing
heuristically and numerically how it may apply to kinetically constrained
models.

\section{\label{sec:Setting_and_results}Setting and main results}
\begin{defn}
\label{def:KCLG}A kinetically constrained lattice gas is a Markov
process on $\Omega=\{0,1\}^{\Z}$ with generator 
\begin{equation}
\cL f(\eta)=\frac{1}{2}\sum_{x,y\in\Z}c(x,y,\eta)(f(\eta^{xy})-f(\eta)),
\end{equation}
where $\eta^{xy}$ is the configuration after exchanging $x$ and
$y$. The rates are given by the \emph{constraint function} $c:\Z\times\Z\times\Omega\to[0,\infty]$
satisfying:
\begin{enumerate}
\item Nearest-neighbor jumps: $c(x,y,\eta)=0$ if $\left|x-y\right|\neq1$.
By convention $c(x,y,\eta)=c(y,x,\eta)$.
\item Translation invariance: $c(x,y,\eta)=c(x+z,y+z,\tau_{z}\eta)$ for
all $x,y,z$, where $\tau_{z}$ is the translation by $z$.
\item $c(x,y,\eta)$ does not depend on the occupation at $x$ and $y$,
i.e., $c(x,y,\eta)=c(x,y,\eta^{xy})$.
\item Finite range: there exists $r>0$ such that $c(x,y,\eta)$ depends
only on the configuration in $x+[-r,r]$.
\item Monotonicity: if $\eta(x)\le\omega(x)$ for all $x$, then $c(x,y,\eta)\ge c(x,y,\omega)$.
That is, adding vacancies speeds up transitions.
\end{enumerate}
We can also define the process on a finite segment $[a,b]$ of $\Z$
by fixing boundary conditions; in this article we always consider
filled boundary, i.e., for $x,y\in[a,b]$ the constraint in the segment
is given by $c_{[a,b]}(x,y,\eta)=c(x,y,\tilde{\eta})$ where $\tilde{\eta}$
agrees with $\eta$ on $[a,b]$ and equals $1$ outside.
\end{defn}

Thanks to the third point of the definition above, for all $q\in(0,1)$
KCLGs are reversible with respect to the Bernoulli product measure
with density $1-q$: 
\begin{equation}
\mu_{q}=\prod_{x\in\Z^{d}}\text{Ber}(1-q).
\end{equation}

\begin{defn}
\label{def:connected}Fix a KCLG (i.e., a constraint function $c$).
Two configurations $\eta,\eta'\in\Omega$ are \emph{connected}, denoted
$\eta\connected{}c\eta'$, if there is a sequence of configurations
$\eta_{0},\eta_{1},\dots,\eta_{n}$ starting with $\eta_{0}=\eta$
and ending at $\eta_{n}=\eta'$, such that $\eta_{i+1}=\eta_{i}^{x,y}$
for $x,y$ such that $c(x,y,\eta_{i})\neq0$. We denote $\eta\connected Lc\eta'$
for configurations connected in the interval $[1,L]$.
\end{defn}

\begin{description}
\item [{Terminology}] We say that a transition $\eta\to\eta^{x,y}$ is
\emph{legal}, or that the constraint is \emph{satisfied}, if $c(x,y,\eta)\neq0$.
A path $\eta_{0},\eta_{1},\dots,\eta_{n}$ is called legal if it consists
of legal transitions. In this language, $\eta\connected{}{}\eta'$
means that there exists a legal path from $\eta$ to $\eta'$.
\end{description}

\begin{defn}
\label{def:qc}A KCLG is said to be \emph{ergodic} at $q$ if $\mu_{q}$-almost
surely $\eta\connected{}{}\eta^{0,1}$. Note that due to the monotonicity
of $c$ adding vacancies helps ergodicity, and we may define the critical
parameter
\begin{eqnarray*}
q_{c} & = & \inf\{q\in(0,1):\text{the process is ergodic for }q\}\\
 & = & \sup\{q\in(0,1):\text{ the process is not ergodic for }q\}.
\end{eqnarray*}
By the same argument as \cite{BertiniToninelli}, this notion of ergodicity
implies that $0$ is a simple eigenvalue of the generator in $L^{2}(\mu_{q})$.
\end{defn}

We are ready to state our main result:
\begin{thm}
\label{thm:q_c}KCLGs in one dimension divide in three universality
classes:
\begin{enumerate}
\item If $c(-1,0,\bone_{(-\infty,-1]})\neq0$ and $c(-1,0,\bone_{[0,\infty]})\neq0$,
then $q_{c}=0$. In this case the model is \emph{noncooperative} (see
\cite{Shapira2024Noncooperative} and Section \ref{subsec:noncooperative}
below).
\item If $c(-1,0,\bone_{(-\infty,-1]})=0$ and $c(-1,0,\bone_{[0,\infty)})=0$,
then $q_{c}=1$. We say in this case that the model is \emph{blocked}.
\item If $c(-1,0,\bone_{(-\infty,-1]})\neq0$ and $c(-1,0,\bone_{[0,\infty)})=0$
or vice versa, then $q_{c}\in(0,1)$. We call such models \emph{one-sided}.
\end{enumerate}
By $\bone_{(-\infty,-1]}$ (respectively $\bone_{[0,\infty)}$) we
mean the configuration filled in $(-\infty,-1]$ (respectively $[0,\infty)$)
and empty elsewhere.
\end{thm}

Finally, we consider the diffusivity of these models. This is already
done in \cite{Shapira2024Noncooperative} for noncooperative models,
showing that the diffusion coefficient is strictly positive. Below
$q_{c}$ (so in particular for all $q$ in blocked models) some edges
are never crossed, so the diffusion coefficient must be $0$. We are
left with one-sided models:
\begin{thm}
\label{thm:D}Consider a one-sided KCLG. Then there exists $q_{0}$
such that for any $q>q_{0}$ the associated diffusion coefficient
is strictly positive.
\end{thm}

\begin{conjecture}
The diffusion coefficient of a KCLG is strictly positive for all $q>q_{c}$.
\end{conjecture}

\section{\label{sec:east_or}The $k$-East-or model}

In this section we study, for fixed $k\in\N$, the (conservative)
$k$-East-or model given by the constraint:
\begin{equation}
c^{\lor}(x-1,x,\eta)=\begin{cases}
1 & \text{if }\eta(x+1)\eta(x+2)\times\dots\times\eta(x+k)=0,\\
0 & \text{otherwise}.
\end{cases}
\end{equation}
That is, the constraint is satisfied (i.e., $c^{\lor}(x-1,x,\eta)\neq0$)
when \emph{at least one} of the $k$ sites to the east is empty.
\begin{prop}
\label{prop:qc_or}Denote the critical parameter of the $k$-East-or
model by $q_{c}^{\lor}$ (see Definition \ref{def:qc}). Then 
\[
q_{c}^{\lor}=\frac{1}{k+1}.
\]
\end{prop}

Before going into the proof, define for $L\in\N$ and $\eta\in\Omega$
\begin{equation}
V_{L}=V_{L}(\eta)=\#\{\text{vacancies in }[1,L]\}=\sum_{x=1}^{L}(1-\eta(x)).
\end{equation}

\subsection{Upper bound}
\begin{lem}
\label{lem:or_move_0}Let $\eta\in\Omega$ and $L\in\N$ such that
$(k+1)V_{L}>L$. Assume $c^{\lor}(-1,0,\eta)=0$. Then there exist
$x\in[2,L]$ such that $\eta(x)=0$, $\eta(x-1)=1$, and $c^{\lor}(x-1,x,\eta)=1$.
\end{lem}

\begin{proof}
Since $c^{\lor}(-1,0,\eta)=0$, the interval $[1,k]$ is completely
filled. Therefore, the interval $[k+1,L]$ contains more than $\frac{L}{k+1}$
vacancies. By the pigeonhole principle, at least two of them must
fall within distance less than $k+1$. We may then take the leftmost
such pair, i.e., among all pairs $(x,y)\in[1,L]$ such that $y\in[x+1,x+k]$
and $\eta(x)=\eta(y)=0$, take one with minimal $x$. We know that
$x>k$ (so in particular $x\in[2,L]$), by minimality $\eta(x-1)=1$,
and $c^{\lor}(x-1,x,\eta)=1$ since $\eta(y)=0$.
\end{proof}
\begin{lem}
Let $\eta\in\Omega$ and $L\in\N$ such that $(k+1)V_{L}>L$. Then
$\eta\connected{}{\lor}\eta^{0,-1}$.
\end{lem}

\begin{proof}
For any configuration $\eta$, we consider its center of mass 
\begin{equation}
\text{com}(\eta)=\sum_{x=1}^{L}x\eta(x).
\end{equation}
Let $\eta'$ be a configuration connected to $\eta$ with maximal
center of mass. That is, $\eta'\connected L{\lor}\eta$ and $\text{com}(\eta'')\le\text{com}(\eta')$
for all $\eta''\connected L{\lor}\eta$. There is no transition in
$\eta'$ legal for the dynamics that moves a particle to the right,
therefore by Lemma \ref{lem:or_move_0} $c^{\lor}(-1,0,\eta')=1$
(since $V_{L}$ doesn't change when exchanging sites in the segment
$[1,L]$).

Since $\eta'\connected L{\lor}\eta$, we may construct a legal path
$\eta=\eta_{0},\eta_{1},\dots,\eta_{n}=\eta'$ (as in Definition \ref{def:connected}).
We then set $\eta_{n+1}=\eta_{n}^{-1,0}$ which is legal since $c^{\lor}(-1,0,\eta')=1$,
and $\eta_{n+1+i}=\eta_{n-i}^{-1,0}$ for $i\le n$ which again describe
legal transitions since the constraint in $[1,L${]} does not depend
on the sites $-1$ and $0$. We conclude that $\eta\connected{}{\lor}\eta_{2n+1}=\eta^{-1,0}$.
\end{proof}
\begin{cor}
$q_{c}^{\lor}\le\frac{1}{k+1}.$
\end{cor}

\begin{proof}
Let $q>\frac{1}{k+1}$. Since $V_{L}\sim\text{Bin}(L,q)$ by the law
of large numbers
\[
\mu_{q}\left[(k+1)V_{L}\le L\right]\xrightarrow{L\to\infty}0,
\]
hence almost surely there exists $L\in\N$ such that $L<(k+1)V_{L}$.
The previous lemma then guarantees $\eta\connected{}{\lor}\eta^{0,1}$.
\end{proof}

\subsection{Lower bound}

For $L\in\N$, define the event 
\begin{equation}
\mathcal{A}_{L}^{\lor}=\left\{ (k+1)V_{l}\le l,\quad\forall l\le L\right\} .
\end{equation}

\begin{lem}
\label{lem:leaving _A_or}Let $\eta\in\mathcal{A}_{L}^{\lor}$ and
$\eta'\connected L{\lor}\eta$. Then $\eta'\in\mathcal{A}_{L}^{\lor}$.
\end{lem}

\begin{proof}
Without loss of generality $\eta'=\eta^{x-1,x}$ for some $x$ such
that $c(x-1,x,\eta)=1$. Assume by contradiction $\eta'\notin\mathcal{A}_{L}^{\lor}$.
Then for some $l\le L$, 
\[
(k+1)V_{l}(\eta)\le l\text{ but }(k+1)V_{l}(\eta')>l.
\]
For $V_{l}$ to grow when exchanging $x-1$ with $x$, necessarily
\[
l=x-1,\quad\eta(x-1)=1,\quad\eta(x)=0.
\]
Since $V_{l}$ grows by one in this case, 
\[
(k+1)V_{l}(\eta)\le l,\quad(k+1)(V_{l}(\eta)+1)>l,
\]
which implies $l-k-1<(k+1)V_{l}\le l.$

We also know that $c(x-1,x,\eta)=1$, so there exist $j\le k$ such
that $x+j\le L$ and $\eta(x+j)=0$. This means that there are at
least two vacancies ($x$ and $x+j$) in the interval $[x,x+j]$,
hence 
\[
V_{l+j+1}\ge V_{l}+2.
\]
This contradicts $\eta\in\mathcal{A}_{L}^{\lor}$, since 
\begin{eqnarray*}
(k+1)Z_{l+j+1} & \ge & (k+1)V_{l}+2(k+1)>l-k-1+2(k+1)\\
 & = & l+k+1\ge l+j+1.
\end{eqnarray*}
\end{proof}
\begin{cor}
$q_{c}^{\lor}\ge\frac{1}{k+1}.$
\end{cor}

\begin{proof}
Fix $q<\frac{1}{k+1}$. We first note that
\begin{equation}
\mu(\mathcal{A}_{L}^{\lor}\text{ for all }L\in\N)=\mu\left(V_{l}-\frac{l}{k+1}\le0,\quad\forall l\in\N\right)>0.
\end{equation}
Indeed, $V_{l}-\frac{l}{k+1}$ is a random walk with IID increments
of mean $q-\frac{1}{k+1}<0$, hence it has positive probability to
remain negative for all $l$.

Fix $\eta\in\cap_{L\in\N}\mathcal{A}_{L}^{\lor}$, and assume $\eta\connected{}{\lor}\eta^{-1,0}$.
Then there must be $L$ large enough such that $\eta\connected L{\lor}\eta'$
and $c^{\lor}(-1,0,\eta')=1$. But $\eta'\in\mathcal{A}_{L}^{\lor}$
by Lemma \ref{lem:leaving _A_or}, and in particular $(k+1)V_{k}(\eta')\le k$
which means $V_{k}(\eta')=0$. This contradicts $c^{\lor}=1$. 
\end{proof}

\section{\label{sec:east_and}The $k$-East-and model}

The (conservative) $k$-East-and model is given, for fixed $k\in\N$,
by the constraint 
\begin{equation}
c^{\land}(x-1,x,\eta)=\begin{cases}
1 & \text{if }\eta(x+1)=\eta(x+2)=\dots=\eta(x+k)=0,\\
0 & \text{otherwise}.
\end{cases}
\end{equation}
That is, the constraint is satisfied when \emph{all} $k$ sites to
the east are empty.
\begin{prop}
\label{prop:qc_and}Denote the critical parameter of the $k$-East-and
model by $q_{c}^{\land}$ (see Definition \ref{def:qc}). Then 
\[
q_{c}^{\land}=\frac{1}{k+1}.
\]
\end{prop}

The proof follows the same steps as for the East-or model; we will
explain it briefly without going too much into the details.

We start with the following notation for the number of filled sites:
\begin{equation}
F_{L}=F_{L}(\eta)=\sum_{x=1}^{L}\eta(x),\qquad\eta\in\Omega,L\in\N.\label{eq:F_L}
\end{equation}

\subsection{Upper bound}
\begin{lem}
\label{lem:and_move_0}Let $\eta\in\Omega$ and $L\in\N$ such that
$(k+1)F_{L}<L$. Assume $c^{\land}(-1,0,\eta)=0$. Then there exist
$x\in[2,L]$ such that $\eta(x)=0$, $\eta(x-1)=1$, and $c^{\land}(x-1,x,\eta)=1$.
\end{lem}

\begin{proof}
The proof follows the same lines as Lemma \ref{lem:or_move_0}: $c^{\land}(-1,0,\eta)=0$
implies that there is at least one filled site in $[1,k]$, hence
there are at most $F_{L}(\eta)-1$ filled sites in $[k+1,L]$. By
the pigeon hole principle there must be an interval of length $k+1$
which is completely empty, and we take the leftmost one, $[x,x+k]$.
By assumption $x>1$, hence $\eta(x-1)=1$, $\eta(x)=0$, and $c^{\land}(x-1,x,\eta)=1$.
\end{proof}
\begin{lem}
Let $\eta\in\Omega$ and $L\in\N$ such that $(k+1)F_{L}<L$. Then
$\eta\connected{}{\land}\eta^{0,-1}$.
\end{lem}

\begin{proof}
As for the East-or model, we may assume $\eta$ to have maximal center
of mass among all $\eta'\connected L{\land}\eta$. Then there is no
legal transition moving a particle to the right, and by Lemma \ref{lem:and_move_0}
$c^{\land}(-1,0,\eta)=1$, which proves the result.
\end{proof}
\begin{cor}
$q_{c}^{\land}\le\frac{k}{k+1}$.
\end{cor}

\begin{proof}
$F_{L}$ is binomial with mean $Lq$, so for $q>\frac{k}{k+1}$, 
\[
\mu\left[F_{L}\ge\frac{L}{k+1}\right]\xrightarrow{L\to\infty}0.
\]
Then there exists almost surely $L\in\N$ such that $(k+1)F_{L}<L$,
and we conclude by the previous lemma.
\end{proof}

\subsection{Lower bound }

For $L\in\N$, define the event 
\begin{equation}
\mathcal{A}_{L}^{\land}=\left\{ (k+1)F_{l}\ge l,\quad\forall l\le L\right\} .\label{eq:def_A_and}
\end{equation}

\begin{lem}
\label{lem:leaving _A_and}Let $\eta\in\mathcal{A}_{L}^{\land}$ and
$\eta'\connected L{\land}\eta$. Then $\eta'\in\mathcal{A}_{L}^{\land}$.
\end{lem}

\begin{proof}
We assume by contradiction and without loss of generality that $\eta'=\eta^{x-1,x}\notin\mathcal{A}_{L}^{\land}$
for $x$ such that $c^{\land}(x-1,x,\eta)=1$. As for the East-or
model, this means that $\eta(x-1)=1,\eta(x)=0,$ and with $l=x-1$
\[
l\le(k+1)F_{l}<l+k+1.
\]

Since $c^{\land}(x-1,x,\eta)=1$, the segment $[x+1,x+k]$ is completely
empty, hence 
\[
(k+1)F_{l+k+1}=(k+1)F_{l}<l+k+1
\]
contradicting $\eta\in\mathcal{A}_{L}^{\land}$.
\end{proof}
\begin{cor}
$q_{c}^{\land}\ge\frac{k}{k+1}$.
\end{cor}

\begin{proof}
The process $F_{l}-\frac{l}{k+1}$ is a random walk with independent
increments of mean $1-q-\frac{1}{k+1}=\frac{k}{k+1}-q$. Hence for
$q<\frac{k}{k+1}$ there is positive probability that $F_{l}-\frac{l}{k+1}\le0$
for all $l$, i.e., $\mu\left[\cap_{L}\mathcal{A}_{L}^{\land}\right]>0$.

For $\eta\in\cap_{L}\mathcal{A}_{L}^{\land}$ we know that $(k+1)F_{k}\ge k$,
so $F_{k}\ge1$ meaning that $c^{\land}(-1,0,\eta)=0$. By the previous
lemma all configurations connected to $\eta$ are also in $\cap_{L}\mathcal{A}_{L}^{\land}$,
so $\eta$ cannot be connected to $\eta^{-1,0}$ (unless $\eta(-1)=\eta(0)$).
\end{proof}

\section{\label{sec:proof_irreducibility}Proof of Theorem \ref{thm:q_c}}

\subsection{\label{subsec:noncooperative}Part 1}

Assume $c(-1,0,\bone_{[-\infty,-1]})\neq0$ and $c(-1,0,\bone_{[0,\infty]})\neq0$.
Recall the range $r$ in Definition \ref{def:KCLG}. We follow the
steps of \cite{BertiniToninelli}.
\begin{lem}
\label{lem:mobile_cluster}The set $[-r,r]$ is a \emph{mobile cluster},
i.e.,
\begin{enumerate}
\item Let $\eta$ such that $\eta(x)=0$ for all $x\in x_{0}+[-r,r]$. Then
there exist $\eta'\connected{}{}\eta$ such that $\eta'(x)=0$ for
all $x\in x_{0}+1+[-r,r]$.
\item Let $\eta$ such that $\eta(x)=0$ for all $x\in x_{0}+[-r,r]$. Then
there exist $\eta'\connected{}{}\eta$ such that $\eta'(x)=0$ for
all $x\in x_{0}-1+[-r,r]$.
\end{enumerate}
\end{lem}

\begin{proof}
Without loss of generality we will focus on point (1) and $x_{0}=0$.
Fix $\eta$ such that $\eta(x)=0$ for all $x\in[-r,r]$. We will
show how to move the cluster to the right.

Note first that $c(r,r+1,\eta)\neq0$: since the range is $r$, we
may change the occupation to the right of $0$ without changing $c(r,r+1,\eta)$.
In particular we are allowed to empty all sites to the right of $0$,
obtaining $c(r,r+1,\eta)=c(r,r+1,\tilde{\eta})$ for 
\[
\tilde{\eta}(x)=\min(\eta(x),\bone_{[r,\infty]}(x)).
\]
By monotonicity of the constraint and translation invariance
\[
c(r,r+1,\eta)\ge c(r,r+1,\bone_{[r,\infty]}(x))\neq0.
\]

We then define $\eta_{1}=\eta^{r,r+1}$. By the same reasoning, letting
$\eta_{i+1}=\eta_{i}^{r-i,r+1-i}$ for $i=1,\dots,r$, we construct
a path connecting $\eta$ to $\eta_{r+1}$, where $\eta_{r+1}(x)=0$
for $x\in[-r-1]\cup[1,r+1]$.

The next step is to continue the path using the vacancies to the right:
we note that $c(-1,0,\eta_{r+1})=c(-1,0,\tilde{\eta}_{r+1})$ for
\[
\tilde{\eta}_{r+1}(x)=\min(\eta(x),\bone_{[-\infty,1]}),
\]
hence by monotonicity and translation invariance $c(-1,0,\eta_{r+1})\neq0$.
Then we may add $\eta_{r+2}=\eta_{r+1}^{-1,0}$, and similarly $\eta_{i+1}=\eta_{i}^{r-i,r+1-i}$
for $i=r+1,\dots,2r$. We end up with a configuration $\eta_{2r+1}$
which is empty on $1+[-r,r]$.
\end{proof}
Models with a mobile cluster are called \emph{noncooperative}, and
were studied in detail in \cite{Shapira2024Noncooperative}. In particular,
it is not difficult to see that noncooperative models are ergodic
for all $q>0$: with probability $1$ there is some empty mobile cluster
to the right of the origin, i.e., there exists $x_{0}>r$ such that
$\eta(x)=0$ for all $x\in x_{0}+[-r,r]$. We may then move it step
by step to the left, until reaching a configuration $\eta'$ where
$[1,2r+2]$ is completely empty. In this configuration we are allowed
to exchange $-1$ and $0$, and then fold the transitions back to
$\eta^{-1,0}$. \qed

\subsection{Part 2}

Assume $c(-1,0,\bone_{[-\infty,-1]})=0$ and $c(-1,0,\bone_{[0,\infty]})=0$.
With positive probability $\eta(x)=1$ for all $x\in[1,r+1]$ ($r$
being the range of the constraint). We claim that $[1,r+1]$ remains
occupied forever: assume that some path empties a site of $[1,r+1]$.
Considering the first step at which this happens we may assume that
there exists $x$ such that $\eta'=\eta^{x-1,x}$, $c(x-1,x,\eta)\neq0$,
and $\eta'$ contains an empty site in $[1,r+1]$. This can only happen
if $x\in\{1,r+2\}$.

Assume $x=1$ (the case $x=r+2$ is analogous). Since the range of
the constraint is $r$, $c(0,1,\eta)=c(0,1,\max(\eta,\bone_{(-\infty,0]}))\le c(0,1,\bone_{(-\infty,0]})=0$,
reaching a contradiction. \qed

\subsection{Part 3}

The proof of this part is based on sections \ref{sec:east_or} and
\ref{sec:east_and}: consider a one-sided model, and assume without
loss of generality $c(-1,0,\bone_{(-\infty,-1]})\neq0$ and $c(-1,0,\bone_{[0,\infty)})=0$.

Since $c(-1,0,\bone_{(-\infty,-1]})\neq0$ and the model has finite
range $r$, for any $x\in\Z$ and $\eta\in\Omega$, if $\eta(x+1)=\dots=\eta(x+r)=0$
then $c(x-1,x,\eta)\neq0$. Therefore, legal transitions for the $r$-East-and
model are legal for our model, so taking $k=r$:
\begin{equation}
\eta\connected{}{\land}\eta'\Rightarrow\eta\connected{}c\eta'.
\end{equation}
This implies $q_{c}\le q_{c}^{\land}.$

In the other direction, $c(-1,0,\bone_{[0,\infty)})=0$ means that
for any $x$ and $\eta$, if $c(x-1,x,\eta)\neq0$ there must be a
vacancy somewhere to the right of $x$, and by finiteness of the range
there must be a vacancy in $[x+1,x+r]$. That is, any legal transition
is also legal for the $r$-East-or model, so if we take $k=r$:
\begin{equation}
\eta\connected{}c\eta'\Rightarrow\eta\connected{}{\lor}\eta.
\end{equation}
Therefore $q_{c}\ge q_{c}^{\lor}.$

By Propositions \ref{prop:qc_or} and \ref{prop:qc_and}, 
\begin{equation}
\frac{1}{r+1}\le q_{c}\le\frac{r}{r+1},
\end{equation}
and in particular $q_{c}$ is non-trivial.\qed

\section{\label{sec:proof_D}Proof of Theorem \ref{thm:D}}

From now on we fix $q>q_{0}$ (for $q_{0}$ to be determined), and
omit the explicit $q$-dependence from the notation.
\begin{observation}
\label{obs:D_and}As explained in the proof of part 3 of Theorem \ref{thm:q_c},
any one-sided model dominates (up to a factor) the $k$-East-and model
for some choice of $k$. By the Green-Kubo formula (see, e.g., \cite[Proposition II.2.2]{Spohn2012IPS})
we may therefore restrict ourselves to proving the theorem for the
$k$-East-and model.
\end{observation}

\subsection{Preliminaries}

The proof of Theorem \ref{thm:D} relies on a variational characterization
of the diffusion coefficient as supremum over flows. The reader is
referred to \cite{Shapira2026Thomson} for detailed explanations,
we will briefly present the necessary input required for this paper.
\begin{defn}
A \emph{flow} is an antisymmetric function $\phi:\Omega\times\Omega\to\R$.
We say that a flow is \emph{admissible} if $\phi(\eta,\eta')\neq0$
implies that there exists $x\in\Z$ such that $\eta'=\eta^{x-1,x}$
and $c(x-1,x,\eta)\neq0$.
\end{defn}

\begin{defn}
\label{def:inner_prod}For an admissible flow $\phi$ we define (when
the sums converge): 
\begin{eqnarray}
\left\langle \phi,\phi\right\rangle  & = & \frac{1}{2}\mu\left[\left(\sum_{x\in\Z}\tau_{x}\phi(\eta,\eta^{-1,0})\right)^{2}\right],\\
\left\langle \phi^{1},\phi\right\rangle  & = & \frac{1}{2}\sum_{x\in\Z}\mu\left[\phi(\eta,\eta^{x-1,x})(\eta(x-1)-\eta(x))\right],
\end{eqnarray}
where $\tau_{x}$ is the translation by $x$.
\end{defn}

\begin{defn}
\label{def:V0}Fix an admissible flow $\phi$. We say that $\phi\in V_{0}$
if 
\begin{equation}
\sum_{x\in\Z}\mu\left[\sum_{\eta'}\phi(\eta,\eta')\times\tau_{x}\sum_{\eta'}\phi(\eta,\eta')\right]=0.\label{eq:div0}
\end{equation}
\end{defn}

\begin{thm}[Corollary 3.8 of  \cite{Shapira2026Thomson}]
\label{thm:thomson}Let $\phi\in V_{0}$ with $\left\langle \phi,\phi\right\rangle \neq0$.
Then 
\[
D\ge\frac{1}{q(1-q)}\frac{\left\langle \phi^{1},\phi\right\rangle ^{2}}{\left\langle \phi,\phi\right\rangle }.
\]
\end{thm}

\begin{rem}
Our goal is to construct a good test flow $\phi$, i.e., an admissible
flow in $V_{0}$, that will have large $\left\langle \phi^{1},\phi\right\rangle $
and small $\left\langle \phi,\phi\right\rangle $. Let us take a more
careful look at the definitions above, to gain better intuition of
what this test flow should look like.

It is instructive to think of $\phi(\eta,\eta')$ as the flow of some
current going from $\eta$ to $\eta'$; positive $\phi(\eta,\eta')$
corresponding to charge transfer from $\eta$ to $\eta'$ and negative
$\phi(\eta,\eta')$ to charge transfer from $\eta'$ to $\eta$. The
flows we should consider are located at the origin, that is, charge
only moves from a configuration $\eta$ to $\eta^{x,x+1}$ for $x$
in a (possibly large) finite segment $\Lambda$. Then the term $\sum_{\eta'}\phi(\eta,\eta')$
in the definition of $V_{0}$ counts the charge exiting $\eta$ (or
negative charge accumulated at $\eta$) for $\phi$. Summing (formally)
over all translations, $\sum_{x\in\Z}\tau_{x}\sum_{\eta'}\phi(\eta,\eta')$
describes the total charge loss, if we account for transitions in
the entire space rather than just near the origin. Thus, $\phi\in V_{0}$
is satisfied when this total charge loss vanishes.

A good example to have in mind are flows of the type $\phi(\eta,\eta^{x-1,x})\propto\eta(x-1)-\eta(x)$:
the flow is positive when particles jump to the left, meaning that
the corresponding charge is related to the center of mass. In this
case, being in $V_{0}$ should be thought of as conservation of the
overall center of mass.

We now look at Definition (\ref{def:inner_prod}). The contribution
to $\left\langle \phi,\phi\right\rangle $ comes from pairs $x,y$
such that $\tau_{x}\phi(\eta,\eta^{0,1})$ is highly correlated with
$\tau_{y}\phi(\eta,\eta^{0,1})$. That is, at the overlap between
the segments $x+\Lambda$ and $y+\Lambda$, the charge transfer ``centered
at $x$'' and the charge transfer ``centered at $y$'' are positively
correlated. The contribution to $\left\langle \phi^{1},\phi\right\rangle $
comes from large (negative\footnote{Since we are only interested in $\left|\left\langle \phi_{l},\phi\right\rangle \right|$
the fact that the correlation is negative is not important, but rather
that it has constant sign.}) correlation between transfer of particles and transfer of charge.

In conclusion, we will be looking for a flow with ``overall conservation
of charge'', that has small overlap between transitions centered
around different points of space, and is highly correlated with the
particle flow.
\end{rem}

\bigskip{}

\subsection{Preparation}

Recall equation (\ref{eq:F_L}), and define
\begin{eqnarray}
l(\eta) & = & \inf\left\{ l:(k+1)F_{l}(\eta)<l\right\} ,\qquad\eta\in\Omega,\\
\Omega_{L} & = & \left\{ \eta\in\Omega:l(\eta)=L\right\} ,
\end{eqnarray}
so $l(\eta)$ is the first $l$ such that $\eta\notin\mathcal{A}_{l}^{\land}$
(see equation (\ref{eq:def_A_and})). By definition 
\begin{equation}
F_{l(\eta)}(\eta)<\frac{l(\eta)}{k+1},
\end{equation}
and we choose $q_{0}$ large enough to guarantee (e.g. by the Chernoff
bound) that
\begin{equation}
\mu\left[l(\eta)=L\right]\le\mu\left[F_{L}<\frac{L}{k+1}\right]\le32^{-L}.\label{eq:prob_bad}
\end{equation}

\subsection{The test flow}

The construction of the flow will be via a path: we will consider
for a configuration $\omega$ a legal path bringing us to $\omega^{-1,0}$.
For each transition we will set our flow to $1$ starting with $\omega(0)=0$
and $\omega(-1)=1$ (i.e., a particle moved to the right); or $-1$
when following the path is the reverse direction. This will make $\left\langle \phi^{1},\phi\right\rangle $
large. The limitation of this estimate, and the reason we cannot take
$q_{0}=q_{c}$, comes from the fact that we have a poor control over
the path itself, making $\left\langle \phi,\phi\right\rangle $ difficult
to bound.

In the following we use the notation 
\begin{equation}
\Omega_{0}=\left\{ \omega\in\Omega:\omega(0)=0,\omega(-1)=1\right\} \cap\left\{ \omega\in\Omega:l(\omega)\text{ finite}\right\} .
\end{equation}

\begin{defn}
\label{def:path}Fix $\omega\in\Omega_{0}$. We construct a legal
path $\omega_{0}=\omega,\omega_{1}=\omega_{0}^{x_{0}-1,x_{0}},\dots,\omega_{2n+1}=\omega_{2n}^{x_{2n}-1,x_{2n}}=\omega^{-1,0}$
(with $n$ depending on $\omega$). We do it in three phases:
\begin{enumerate}
\item Begin with a path $\omega_{0}=\omega,\omega_{1}=\omega_{0}^{x_{0}-1,x_{0}},\dots,\omega_{n}=\omega_{n-1}^{x_{n-1}-1,x_{n-1}}$,
whose last configuration satisfies $c(-1,0,\omega_{n})\neq0$ (here
$c=c^{\land}$ by Observation \ref{obs:D_and}). Recall Lemma \ref{lem:and_move_0},
and that $(k+1)F_{l(\omega)}(\omega)<l(\omega)$. If $c(-1,0,\omega)\neq0$
we are done (taking $n=0$). Otherwise, there exists $x\in[2,l(\omega)]$
with $\omega(x)=0,\ \omega(x-1)=1,\ c(x-1,x,\omega)=1$. Set $x_{0}$
to be the minimal such $x$, and $\omega_{1}=\omega^{x_{0}-1,x_{0}}$.

Continue by induction: assume that we constructed $\omega_{m}$. If
$c(-1,0,\omega)\neq0$ we move to phase 2. Otherwise, set $x_{m}$
to be the minimal $x\in[2,l(\omega)]$ such that $\omega_{m}(x)=0,\ \omega_{m}(x-1)=1,\ c(x-1,x,\omega_{m})=1$,
and $\omega_{m+1}=\omega_{m}^{x_{m}-1,x_{m}}$.

This procedure is valid, since we only exchange sites in $[1,l(\omega)]$,
so the condition $(k+1)F_{l(\omega)}(\omega_{m})<l(\omega)$ holds
for all $m$. Moreover, it must end: at each step the center of mass
$\sum_{x=1}^{l(\omega)}x\eta(x)$ increases by $1$, and it cannot
exceed $\sum_{x=1}^{l(\omega)}x\cdot1\le l(\omega)^{2}.$
\item Next, exchange $-1$ and $0$, choosing $x_{n}=0$ and $\omega_{n+1}=\omega_{n}^{-1,0}$.
\item Finally, wind back the path of phase 1: for $m\in\{n+1,\dots2n\}$
we set $x_{m}=x_{2n-m}$, and $\omega_{m+1}=\omega_{m}^{x_{m}-1,x_{m}}$.
Since all these transitions are in $[1,L]$ they do not depend on
the configuration at $-1$ and $0$, hence they are all legal and
end at $\omega_{2n+1}=\omega^{-1,0}$.
\end{enumerate}
\end{defn}

It is not difficult to verify the following properties of the path:
\begin{lem}
For every $\omega\in\Omega_{0}$, the path given in Definition \ref{def:path}
satisfies:
\begin{enumerate}
\item $\omega_{m}(x)=\omega(x)$ for all $m$ and all $x\notin[-1,l(\omega)]$.
\item $\omega_{m}(-1)=1,\ \omega_{m}(0)=1$ for $m\in\{0,\dots,n\}$, and
$\omega_{m}(-1)=0,\ \omega_{m}(0)=1$ for $m\in\{n+1,\dots,2n+1\}$.
\item $F_{l(\omega)}(\omega_{m})<\frac{l(\omega)}{k+1}$ for all $m$.
\item $x_{m}\le l(\omega)$ for all $m$.
\end{enumerate}
\end{lem}

We will now define our flow using these paths:
\begin{eqnarray}
\psi(\eta,\eta') & = & \sum_{\omega\in\Omega_{0}}\psi_{\omega}(\eta,\eta'),\\
\psi_{\omega}(\eta,\eta') & = & \sum_{j=1}^{2n+1}\One_{\eta=\omega_{j-1},\eta'=\omega_{j}}-\sum_{j=1}^{2n+1}\One_{\eta=\omega_{j},\eta'=\omega_{j-1}}.
\end{eqnarray}
By construction $\psi(\eta',\eta)=-\psi(\eta,\eta')$, and $\psi$
is supported on legal transitions, making it an admissible flow. Note
that, while $\Omega_{0}$ is non-countable, the summand $\psi_{\omega}(\eta,\eta')$
vanishes almost surely for all but finitely many configurations: for
$\eta\in\Omega$, let
\begin{equation}
\Omega_{0}(\eta)=\left\{ \omega\in\Omega_{0}:\exists m\in\{0,\dots,2n(\omega)+1\}\text{ such that }\omega_{m}=\eta\right\} ,
\end{equation}
and observe
\[
\forall\omega\in\Omega_{0}(\eta),\quad F_{l(\omega)}(\eta)<\frac{l(\omega)}{k+1}.
\]
By equation (\ref{eq:prob_bad}), $\sum_{L}\mu\left[F_{L}<\frac{L}{k+1}\right]<\infty$
, so almost surely $F_{L}<\frac{L}{k+1}$ finitely often. We may then
define
\[
L_{*}=\max\left\{ L:F_{L}(\eta)<\frac{L}{k+1}\right\} <\infty.
\]
Then $l(\omega)\le L_{*}$, and therefore $\omega$ agrees with $\eta$
outside $[-1,L_{*}]$, so 
\[
\left|\Omega_{0}(\eta)\right|\le2^{L_{*}+2}
\]
and the sum in the definition of $\psi$ is indeed finite for almost
all $\eta\in\Omega$.
\begin{lem}
For almost all $\eta\in\Omega$, 
\[
\sum_{\eta'}\psi(\eta,\eta')=\eta(-1)-\eta(0).
\]
\end{lem}

\begin{proof}
Fix $\eta\in\Omega$, and let $\omega\in\Omega_{0}(\eta)$. If $\eta$
is some intermediate point of the path, i.e., $\eta=\omega_{m}$ for
$m\notin\{0,2n+1\}$, then $\sum_{\eta'}\psi_{\omega}(\eta,\eta')$
will have two contributions with opposite signs, from $\eta'=\omega_{m+1}$
and from $\eta'=\omega_{m-1}$. Since these cancel out, we are left
with either $\eta=\omega$ or $\eta=\omega_{2n+1}=\omega^{-1,0}$.
That is,
\begin{eqnarray*}
\sum_{\eta'}\psi(\eta,\eta') & = & \sum_{\omega\in\Omega_{0}(\eta)}\One_{\eta=\omega}-\sum_{\omega\in\Omega_{0}(\eta)}\One_{\eta=\omega^{-1,0}}=\One_{\eta\in\Omega_{0}}-\One_{\eta^{-1,0}\in\Omega_{0}}.
\end{eqnarray*}
If $\eta(-1)-\eta(0)=1$, then $\eta\in\Omega_{0}$ and $\eta^{-1,0}\notin\Omega_{0}$;
if $\eta(-1)-\eta(0)=0$ then $\eta\notin\Omega_{0}$ and $\eta^{-1,0}\notin\Omega_{0}$,
and if $\eta(-1)-\eta(0)=-1$ then $\eta\notin\Omega_{0}$ and $\eta^{-1,0}\in\Omega_{0}$.
This proves the lemma.
\end{proof}
\begin{lem}
$\psi\in V_{0}$.
\end{lem}

\begin{proof}
By the discussion above $\psi$ is an admissible flow. For $N\in\N$,
\begin{multline*}
\sum_{x\in[-N,N]}\mu\left[(\eta(-1)-\eta(0))\ \tau_{x}\sum_{\eta'}\psi(\eta,\eta')\right]=\mu\left[(\eta(-1)-\eta(0))\ \sum_{x\in[-N,N]}\left(\eta(x-1)-\eta(x)\right)\right]\\
=\mu\left[(\eta(-1)-\eta(0))\ \left(\eta(-N-1)-\eta(N)\right)\right].
\end{multline*}
The last term equals $0$ for $N>0$, hence equation (\ref{eq:div0})
is satisfied.
\end{proof}
\begin{lem}
\label{lem:psi_psi}$\left\langle \psi,\psi\right\rangle \le11$.
\end{lem}

\begin{proof}
Define the event 
\begin{equation}
\mathcal{B}(x,L)=\left\{ \eta:F_{L}(\tau_{x}\eta)<\frac{L}{k+1}\right\} ,
\end{equation}
so by equation (\ref{eq:prob_bad}) and translation invariance of
$\mu$
\begin{equation}
\mu\left[\mathcal{B}(x,L)\right]\le32^{-L}.
\end{equation}

Let us estimate:
\begin{eqnarray*}
\left|\sum_{x\in\Z}\tau_{x}\psi(\eta,\eta^{-1,0})\right| & = & \left|\sum_{x=0}^{\infty}\psi(\tau_{x}\eta,(\tau_{x}\eta)^{x-1,x})\right|\\
 & \le & =\sum_{x=0}^{\infty}\sum_{\omega\in\Omega_{0}(\tau_{x}\eta)}\One_{x\le l(\omega)}=\sum_{x=0}^{\infty}\sum_{L=x}^{\infty}\sum_{\omega\in\Omega_{0}(\tau_{x}\eta)}\One_{l(\omega)=L}\One_{\eta\in\mathcal{B}(x,L)}\\
 & \le & \sum_{x=0}^{\infty}\sum_{L=x}^{\infty}2^{L+2}\One_{\mathcal{B}(x,L)}=\sum_{L=1}^{\infty}2^{L+2}\sum_{x=0}^{L}\One_{\eta\in\mathcal{B}(x,L)}.
\end{eqnarray*}
Then
\begin{eqnarray*}
\left\langle \psi,\psi\right\rangle  & \le & =\frac{1}{2}\mu\left[\left(\sum_{L=1}^{\infty}2^{L+2}\sum_{x=0}^{L}\One_{\eta\in\mathcal{B}(x,L)}\right)^{2}\right]=2\mu\left[\left(\sum_{L=1}^{\infty}2^{-L}\ 4^{L}\sum_{x=1}^{L}\One_{\eta\in\mathcal{B}(x,L)}\right)^{2}\right]\\
 & \le & 2\mu\left[\sum_{L=1}^{\infty}4^{-L}\times\sum_{L=1}^{\infty}16^{L}\left(\sum_{x=0}^{L}\One_{\eta\in\mathcal{B}(x,L)}\right)^{2}\right]\le\mu\left[\sum_{L=1}^{\infty}16^{L}(L+1)\sum_{x=0}^{L}\One_{\eta\in\mathcal{B}(x,L)}\right]\\
 & = & \sum_{L=1}^{\infty}16^{L}(L+1)\sum_{x=0}^{L}\mu\left[\mathcal{B}(x,L)\right]\le\sum_{L=1}^{\infty}(L+1)^{2}\ 2^{-L}=11
\end{eqnarray*}
\end{proof}
The next step will be to estimate $\left\langle \phi^{1},\psi\right\rangle $,
and we start with the following lemma:
\begin{lem}
Fix $\eta\in\Omega$, and let $\tilde{\eta}=\eta^{-1,0}$. Then for
all $x\ge2$, 
\[
\psi(\eta,\eta^{x-1,x})=-\psi(\tilde{\eta},\tilde{\eta}^{x-1,x}).
\]
\end{lem}

\begin{proof}
Recall the construction of the path (Definition \ref{def:path}).
We note first that $\Omega_{0}(\eta)=\Omega_{0}(\tilde{\eta})$ since
if $\eta=\omega_{m}$ for some $\omega\in\Omega_{0}$ then $\tilde{\eta}=\omega_{\tilde{m}}$
for $\tilde{m}=2n+1-m$. Moreover, if $\eta^{x-1,x}=\omega_{m+1}$
then $\tilde{\eta}^{x-1,x}=\omega_{\tilde{m}-1}$ and vice versa,
hence $\psi_{\omega}(\eta,\eta^{x-1,x})=-\psi_{\omega}(\tilde{\eta},\tilde{\eta}^{x-1,x})$.
This proves the lemma.
\end{proof}
\begin{lem}
\label{lem:phi1_psi}$\left\langle \phi^{1},\psi\right\rangle \ge\frac{1}{2}(1-q)q^{k}$.
\end{lem}

\begin{proof}
Using the previous lemma and the invariance of the measure under the
mapping $\eta\mapsto\eta^{-1,0}$,
\begin{multline*}
\sum_{x\ge2}\mu\left[\psi(\eta,\eta^{x-1,x})(\eta(x-1)-\eta(x))\right]=\sum_{x\ge2}\mu\left[\psi(\eta^{-1,0},\eta^{-1,0;x-1,x})(\eta^{-1,0}(x-1)-\eta^{-1,0}(x))\right]\\
=\sum_{x\ge2}\mu\left[-\psi(\eta,\eta^{x-1,x})(\eta(x-1)-\eta^{-1,0}(x))\right]=0.
\end{multline*}
Recalling Definition \ref{def:inner_prod} and the fact that all transitions
along the path are at $x\in\{0\}\cup[2,\infty)$, we obtain
\begin{eqnarray*}
\left\langle \phi^{1},\psi\right\rangle  & = & \frac{1}{2}\mu\left[\psi(\eta,\eta^{-1,0})(\eta(-1)-\eta(0))\right].
\end{eqnarray*}

Fix $\eta\in\Omega$ and $\omega\in\Omega_{0}(\eta)$. Then $\psi_{\omega}(\eta,\eta^{-1,0})\neq0$
means that either $\eta=\omega_{n}$ or $\eta=\omega_{n+1}$. In the
first case, $\psi_{\omega}(\eta,\eta^{-1,0})=1$ and $\eta(-1)-\eta(0)=1$,
and in the second $\psi_{\omega}(\eta,\eta^{-1,0})=-1$ and $\eta(-1)-\eta(0)=-1$.
Therefore
\[
\left\langle \phi^{1},\psi\right\rangle =\frac{1}{2}\mu\left[\sum_{\omega\in\Omega_{0}(\eta)}\One_{\eta\in\{\omega_{n},\omega_{n+1}\}}\right].
\]
Finally, we note that if $\eta(-1)=1$ and $\eta(0)=\eta(1)=\dots=\eta(k)=0$,
then the configuration $\omega=\eta$ is associated with a path given
by a single transition $\omega_{0}=\eta,\omega_{1}=\eta^{-1,0}$.
Therefore $\omega\in\Omega_{0}(\eta)$ and $\eta\in\{\omega_{n},\omega_{n+1}\}$,
and
\[
\left\langle \phi^{1},\psi\right\rangle \ge\frac{1}{2}\mu\left[\eta(-1)=1,\ \eta(0)=\eta(1)=\dots=\eta(k)=0\right].
\]
\end{proof}
We can now plug Lemma \ref{lem:psi_psi} and Lemma \ref{lem:phi1_psi}
into Theorem \ref{thm:thomson} to obtain a lower bound of $\frac{1}{44}(1-q)q^{2k-1}$
on the diffusion coefficient. \qed

\section{\label{sec:SOC}Remark on dynamical phase transition and self-organized
criticality}

Dynamical phase transitions often come together with self-organized
criticality (see, e.g., \cite{Rolla20ARW}): imagine that we start
with a completely empty configuration on the line $[1,L]$, and remove
vacancies (equivalently inject particles) at the boundary. In the
beginning vacancies move, occasionally reaching the boundary and being
replaced by a particle. This happens until reaching the critical density.
Then, diffusion stops, the remaining vacancies cannot reach the boundary,
and the system fixates at its critical value.

Another setting where one sees a self-organized critical state is
when starting on the infinite line with $N_{e}$ empty sites, in $[L-N_{e},L]$,
and particles elsewhere. Then vacancies disperse in space, until reaching
the critical density where they can no longer move, creating an interval
where the system is at criticality.

\begin{figure}
\begin{tabular}{cc}
\includegraphics[width=0.4\textwidth]{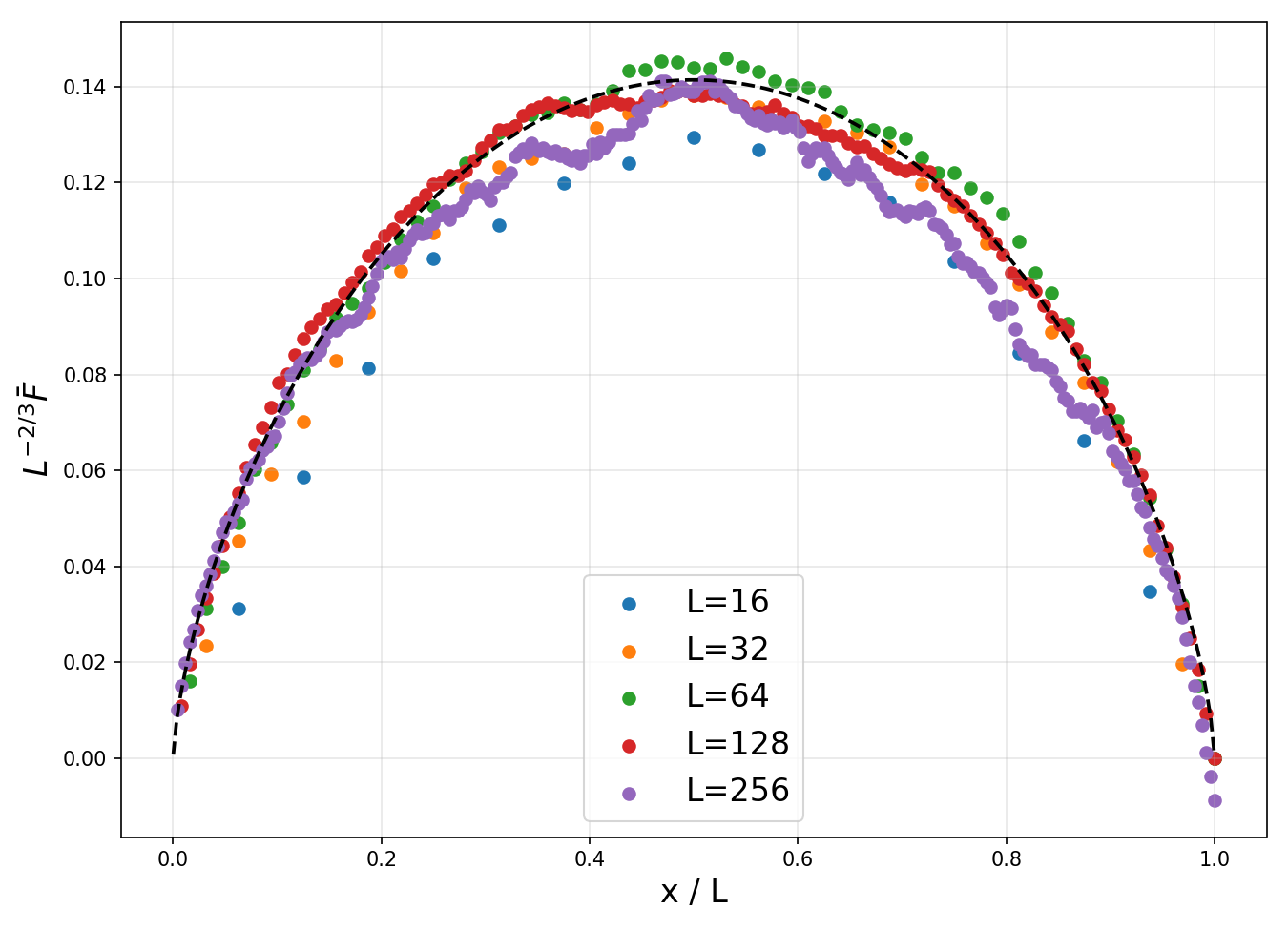} & \includegraphics[width=0.4\textwidth]{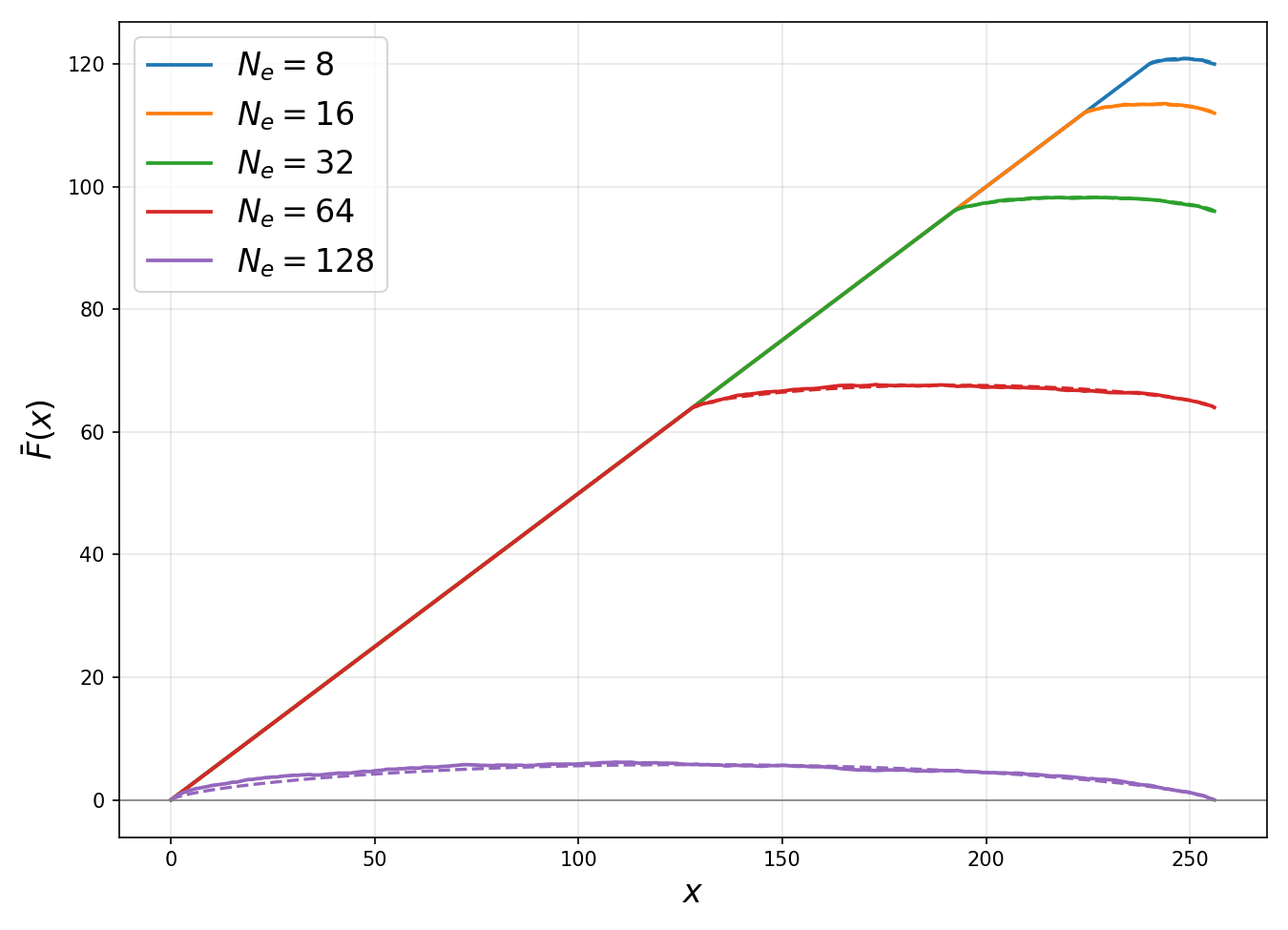}\tabularnewline
(a) & (b)\tabularnewline
\end{tabular}

\caption{\label{fig:sim_soc}The centered cumulative occupation for the conserved
East model with $k=1$, see Section \ref{sec:SOC}. (a) shows the
collapse of $\overline{F}$ for different values of $L$, when inserting
particles from the left. (b) shows, for $L=256$, the centered cumulative
occupation for different values of $N_{e}$. In both figures, the
dashed line is the fit given by equation (\ref{eq:Fbar_fit}), with
recentering in (b).}

\end{figure}

Figure \ref{fig:sim_soc} shows simulations for the conserved East
model with $k=1$ (where the ``or'' and ``and'' versions are the
same). In this case the critical density is $\frac{1}{2}$. In both
settings we reach the critical density $0.5$, plus a correction that
seems to scale as $L^{-1/3}$.

In order to see that, we analyze (for a system on $[1,L]$) the centered
cumulative occupation:
\[
\overline{F}(x)=\sum_{y=1}^{x}(\eta(y)-\frac{1}{2}).
\]

The results of the first experiment, where particles are inserted
at the (left) boundary, are given in Figure \ref{fig:sim_soc}(a).
Simulations suggest an $L^{2/3}$ scaling, with the fit: 
\begin{equation}
\overline{F}(x)=0.36\times L^{2/3}\left(\frac{x}{L}(1-\frac{x}{L})\right)^{2/3}.\label{eq:Fbar_fit}
\end{equation}
This means that the density $\rho(x)=\frac{1}{2}+\overline{F}'(x)=\frac{1}{2}+0.24\times L^{-1/3}\left(\frac{x}{L}(1-\frac{x}{L})\right)^{-1/3}\left(1-\frac{2x}{L}\right)$
is critical, with a correction $O(L^{-1/3})$. We can also see that
for $x\ll L$ (or $L-x\ll L$) this converges to an $L$-independent
function, $\rho(x)\approx\frac{1}{2}+0.24x^{-1/3}$, suggesting an
infinite limiting state with density $\frac{1}{2}$ plus a boundary
effect decaying as $x^{-1/3}$.

The second experiment, starting with vacancies in $[L-N_{e},L]$ on
the infinite line, is shown in Figure \ref{fig:sim_soc}(b). Due to
the constraint, vacancies stay to the left of $L$, and Lemma \ref{lem:leaving _A_or}
guarantees that they will never go to the left of $L-2N_{e}$. Hence
we restrict to the finite segment $[1,L]$, taking $N_{e}\le\frac{L}{2}$.
Then $\overline{F}$ is (deterministically) increasing with slope
$\frac{1}{2}$ up to $L-2N_{e}$, and then becomes flat, corresponding
to density $\frac{1}{2}$ in the interval $[L-2N_{e},L]$. We see
this phenomenon in Figure \ref{fig:sim_soc}(b), with the exact same
$L^{2/3}$ correction of equation (\ref{eq:Fbar_fit}) ($2N_{e}$
playing the role of $L$).

While this self-organized criticality has some similarities with those
observed in sandpile-type models (e.g. \cite{Rolla20ARW,ErignouxShapiraSimon2026CLG}),
it is of a somewhat different nature: the process is reversible, and
the critical state reached is not an absorbing configuration. It is
therefore unclear to which universality class it belongs, whether
the exponent measured above is really an exact $\frac{2}{3}$, and
why.

\section*{Acknowledgements}

I would like to thank Cristina Toninelli and Oriane Blondel for the
useful discussions. The code for the simulation discussed in Section
\ref{sec:SOC} was written with the assistance of Claude (Anthropic).

\bibliographystyle{plain}
\bibliography{1d_kclg}

\end{document}